\documentclass[12pt]{amsart}
\usepackage{amsfonts}
\usepackage{amssymb}
\usepackage{amsmath,amscd}
\input xy
\xyoption {all}

\usepackage{geometry}
\usepackage{amsthm}
\usepackage{amssymb}
\usepackage{latexsym}
\usepackage[dvips]{graphics}

\newtheorem{theorem}{Theorem}[section]

\newtheorem{cor}[theorem]{Corollary}
\newtheorem{defn}[theorem]{Definition}

\newtheorem{lemma}[theorem]{Lemma}
\newtheorem{prop}[theorem]{Proposition}
\newtheorem{remark}[theorem]{Remark}
\newtheorem*{ack}{Acknowledgements}

\newcommand{\mh}{\mbox{MHM}}
\newcommand{\ms}{\mbox{MHS}}
\newcommand{\mc}{\mbox{MHC}_y}

\newcommand{\co}{\mbox{Coh}}
\newcommand{\cs}{\mbox{Constr}}
\newcommand{\vct}{\mbox{Vect}}

\newcommand{\Xs}{{X^{(n)}}}
\newcommand{\DC}{{\mathcal D}}
\newcommand{\LL}{{\mathcal L}}

\newcommand{\E}{{\mathcal E}}

\newcommand{\OO}{{\mathcal O}}
\newcommand{\HC}{{\mathcal H}}

\newcommand{\FC}{{\mathcal F}}
\newcommand{\GC}{{\mathcal G}}
\newcommand{\MC}{{\mathcal M}}
\newcommand{\PC}{{\mathcal P}}
\newcommand{\IC}{{\mathcal I}}

\newcommand{\Z}{\mathbb{Z}}

\newcommand{\Q}{\mathbb{Q}}

\newcommand{\C}{\mathbb{C}}

\begin{document}

\title[Classes of symmetric products]{Characteristic classes of symmetric products\\ of complex quasi-projective varieties}

\dedicatory{\it To the memory of Friedrich Hirzebruch}

\author[S. E. Cappell ]{Sylvain E. Cappell}
\address{S. E. Cappell: Courant Institute, New York University, 251 Mercer Street, New York, NY 10012, USA}
\email {cappell@cims.nyu.edu}

\author[L. Maxim ]{Laurentiu Maxim}
\address{L. Maxim : Department of Mathematics, University of Wisconsin-Madison,
480 Lincoln Drive,
Madison, WI 53706-1388, USA.}
\email {maxim@math.wisc.edu}

\author[J. Sch\"urmann ]{J\"org Sch\"urmann}
\address{J.  Sch\"urmann : Mathematische Institut,
          Universit\"at M\"unster,
          Einsteinstr. 62, 48149 M\"unster,
          Germany.}
\email {jschuerm@math.uni-muenster.de}

\author[J. L. Shaneson ]{Julius L. Shaneson}
\address{J. L. Shaneson: Department of Mathematics, University of Pennsylvania, 209 S 33rd St., Philadelphia, PA 19104, USA}
\email {shaneson@sas.upenn.edu}

\author[S. Yokura ]{Shoji Yokura}
\address{S. Yokura : Department of Mathematics and Computer Science, Faculty of Science, Kagoshima University, 21-35 Korimoto 1-chome, Kagoshima 890-0065, Japan.}
\email {yokura@sci.kagoshima-u.ac.jp}


\date{\today}


\begin{abstract}  We prove generating series formulae for suitable twisted characteristic classes of symmetric products of a {\it singular} complex quasi-projective variety. More concretely, we study  homology Hirzebruch classes for motivic coefficients, as well as for complexes of mixed Hodge modules. As a special case, we obtain a generating series formula for the (intersection) homology Hirzebruch classes of symmetric products. In some cases, the latter yields a similar formula for twisted homology $L$-classes generalizing results of Hirzebruch-Zagier and Moonen. Our methods also apply to the study of Todd classes of (complexes of) coherent sheaves, as well as Chern classes of (complexes of) constructible sheaves, generalizing to arbitrary coefficients results of Moonen and resp. Ohmoto.

\end{abstract}

\maketitle

\tableofcontents


\section{Introduction}

Some of the most interesting examples of orbifolds are the symmetric products of algebraic varieties. The {\it $n$-th symmetric product} of a space $X$ is defined by
$${X^{(n)}}:=\overbrace{X \times \cdots \times X}^{n \ times}/{\Sigma_n},$$
i.e., the quotient of the product of $n$ copies of $X$ by the natural action of the symmetric group  on $n$ elements, $\Sigma_n$.  

The standard approach for computing invariants ${\IC}(X^{(n)})$ of symmetric products is to encode the respective invariants of {\it all} symmetric products in a {\it generating series}, i.e., an expression of the form $$S_{\IC}(X):=\sum_{n \geq 0} \mathcal{I}( X^{(n)} ) \cdot t^n,$$ 
provided  $ \mathcal{I}( X^{(n)} ) $ can be defined for all $n$.
This is analogous  to the zeta function of a variety over a finite field. The aim is to calculate such an expression solely in terms of invariants of $X$, so  $\mathcal{I}( X^{(n)} )$ is just the coefficient of $t^n$ in the resulting expression.

There is a well-known formula due to Macdonald \cite{Mac} for the generating series of the topological Euler characteristic. A  class version of this formula was recently obtained by Ohmoto in \cite{Oh} for the Chern classes of MacPherson \cite{MP}. Moonen \cite{M} obtained generating series for the arithmetic genus of symmetric products of a  projective manifold and, more generally, for the Todd classes of symmetric products of any projective variety, as well as for his generalized Todd classes $\tau_y$ (which he could only define for a projective orbifold $Y/G$, with $G$ a finite group of algebraic automorphisms acting on a projective manifold $Y$). In \cite{Za}, Hirzebruch and Zagier obtained such generating series for the signature and $L$-classes of symmetric products of rational homology manifolds. More precise $L$-class formulae for symmetric products of projective manifolds were obtained by Moonen  \cite{M} as a specialization of his generalized Todd class formula at $y=1$. 
Also, Borisov--Libgober \cite{BL3} computed generating series for the Hirzebruch $\chi_y$-genus and, more generally, for elliptic genus of symmetric products of smooth compact varieties. Generating series for the mixed Hodge numbers of complexes of mixed Hodge modules on symmetric products of (possibly singular) quasi-projective varieties have been recently obtained in \cite{MS09} by relating symmetric group actions on external products to the theory of lambda rings (e.g., see \cite{Yau}), but see also \cite{MSS} for an alternative approach.

\bigskip

In this paper we assume that $X$ is a (possibly singular) complex quasi-projective variety, so its symmetric products  $X^{(n)}$ are quasi-projective varieties as well. 

The invariants of symmetric products considered in this paper are homology characteristic classes such as the un-normalized Hirzebruch classes ${T_y}_*$ of Brasselet--Sch\"urmann--Yokura \cite{BSY}, (rationalized) Chern classes $c_*$ of MacPherson \cite{MP}, and the Todd classes $td_*$ of Baum-Fulton-MacPherson \cite{BFM} and Moonen \cite{M}. In fact, we derive our formulae for characteristic classes with coefficients such as (complexes of) mixed Hodge modules, constructible or coherent sheaves. In particular, for $X$ pure-dimensional, we obtain formulae for twisted intersection Hirzebruch classes ${IT_y}_*$ 
as studied by Cappell--Maxim--Shaneson \cite{CMS}. Conjecturally, for a projective variety $X$, the specialization of these  twisted intersection Hirzebruch classes at $y=1$ should equal (up to re-normalization) the Cappell-Shaneson homology $L$-class (cf.  \cite{CS}) of the underlying self-dual twisted intersection cohomology classes. This identification is known in some cases when $X$ is smooth and projective.

\bigskip

\noindent\emph{Convention.} To simplify the notation, we will often write $cl_*(\PC):=cl_*([\PC])$ for the value of a characteristic class transformation $cl_*(-)$  on the class of $\PC$ in a suitable  Grothendieck group.


\subsection{Motivic Hirzebruch classes}\label{motHir} We first explain our results for the  un-normalized {\it motivic} Hirzebruch class transformation (see \cite {BSY, SY, Yokura-MSRI}): 
$${T_y}_*:K_0(var/X) \to H_{ev}^{BM}(X) \otimes \Q[y],$$ where $H_{ev}^{BM}(-)$ denotes the Borel--Moore homology in even degrees, and  $K_0(var/X)$ is the relative Grothendieck group of complex algebraic varieties over $X$ (as introduced by Looijenga \cite{Lo}), which is freely generated by isomorphism classes of morphisms $[Y \to X]$ modulo the ``scissor" relation
\begin{equation}\label{sc} [Y\to X]=[Z \hookrightarrow Y \to X] + [Y \setminus Z \hookrightarrow Y \to X], \end{equation}
for $Z$  a Zariski closed subset in $Y$. 
If we let  $Z = Y_{red}$ we deduce that these classes $[Y \to X]$ depend only on the underlying reduced spaces. By resolution of singularities, $K_0(var/X)$ is generated by classes $[Y \to X]$ with $Y$ smooth, pure
dimensional, and proper over $X$. If $X$ is a point space, we get the motivic Grothendieck group $K_0(var/\C)$, and the motivic Hirzebruch class  transformation $T_{y*}$  reduces in this case  to a ring homomorphism $\chi_y:K_0(var/\C) \to \Z[y].$

The un-normalized {\it motivic Hirzebruch class} and resp. the {\it $\chi_y$-genus} of a complex algebraic variety $X$ are defined by:
$$T_{y*}(X):=T_{y*}([id_X]) \quad \text{and} \quad \chi_{y}(X):=T_{y*}([X\to pt]).$$
We  point out that the homology class ${T_y}_*(X)$ is an extension to the singular setting of the un-normalized cohomology Hirzebruch class $T_y^*(-)$
appearing in the generalized Hirzebruch--Riemann--Roch theorem \cite{H}, which in Hirzebruch's philosophy corresponds to the un-normalized  power series $Q_y(\alpha)=\frac{\alpha (1+ye^{-\alpha})}{1-e^{-\alpha}} \in \Q[y][[\alpha]]$.  More precisely, for $X$ smooth, we have
$$T_{y*}(X)=T_y^*(TX) \cap [X]:=\Big( \prod_{i=1}^{\dim(X)} Q_y(\alpha_i) \Big) \cap [X],$$ with $\alpha_i$ the Chern roots of the tangent bundle $TX$.
The associated normalized power series (needed in \S \ref{r1}) is $\widehat{Q}_y(\alpha):=\frac{Q_y \left( \alpha (1+y) \right)}{1+y}=\frac{\alpha (1+y)}{1-e^{-\alpha (1+y)}}-\alpha y$, which defines the  normalized cohomology Hirzebruch class ${\widehat{T}_y}^*(-)$. If we specialize the parameter $y$ of ${\widehat{T}_y}^*(-)$ to the three distinguished values $y=-1, 0$ and $1$, we recover the cohomology Chern, Todd, and L-class, respectively. 

\bigskip

For any morphism $f:X' \to X$ we have a functorial push-forward 
$$f_!:K_0(var/X') \to K_0(var/X) \ , \ [Z \overset{h}{\to} X'] \mapsto [Z \overset{f \circ h}{\to} X].$$
Moreover, an external product
$$ \boxtimes : K_0(var/X) \times K_0(var/X') \to K_0(var/X\times X')$$ is defined by the formula:
$$[Z \to X] \boxtimes [Z' \to X'] =[Z \times Z' \to X \times X'].$$
These two structures on the relative motivic Grothendieck group allow us to introduce a corresponding motivic Pontrjiagin ring similar to a well-known construction in homology. More generally, 
let $F$ be a functor to the category of abelian groups, defined on complex quasi-projective varieties, covariantly functorial for all (proper) morphisms. Assume $F$ is also endowed with a commutative, associative and bilinear cross-product $\boxtimes$ commuting with (proper) push-forwards $(-)_*$, with a unit $1\in F(pt)$ for $pt$ denoting the point space. The only examples for $F(X)$ needed in this paper  are:  the relative motivic Grothendieck group $K_0(var/X)$ and  the Borel-Moore homology  $H_*(X):=H^{BM}_{even}(X) \otimes R$ with coefficients $R=\Q$ or $\Q[y]$.

\begin{defn} For a fixed complex quasi-projective variety $X$ 
we define the commutative Pontrjagin ring $\left(PF(X), \odot \right)$ by
$$PF(X):=\sum_{n=0}^{\infty}  F(\Xs) \cdot t^n 
:=\prod_{n=0}^{\infty}  F(\Xs)
,$$
with product $\odot$ induced via
$$\odot:F(\Xs) \times F(X^{(m)}) \overset{\boxtimes}{\to} F(\Xs \times X^{(m)}) \overset{(-)_*}{\to} F(X^{(n+m)}),$$
and unit $1 \in F(X^{(0)})=F(pt)$.
\end{defn}

It is easy to see that, if $f:Y \to X$ is a (proper) morphism, then we get an induced ring homomorphism $$f_*:=(f^{(n)}_*)_n : PF(Y) \to PF(X),$$ with $f^{(n)}:Y^{(n)} \to \Xs$ the corresponding (proper) morphism on the $n$-th symmetric products. In our examples, properness of morphisms is needed for the case $F=H_*$ of Borel-Moore homology.

We can now state our first result:
\begin{theorem}\label{t0}
Let  $f:Y \to X$ be a morphism of complex quasi-projective varieties. Then  
the following identity holds in the Pontrjagin ring $PH_*(X)$:
\begin{equation}\label{te0}
\sum_{n \geq 0} {{T_{(-y)}}_*} ([f^{(n)}]) \cdot t^n= \exp \left( \sum_{r \geq 1} \Psi_r  \Big(d^r_* {{T_{(-y)}}_*}([f]) \Big) \cdot \frac{t^r}{r} \right),\end{equation}
where
\begin{enumerate}
\item[(a)] $d^r:X \to X^{(r)}$ 
is the composition of  the diagonal embedding $i_r:X\simeq \Delta_r(X) \hookrightarrow X^r$ with the projection $\pi_r: X^r \to X^{(r)}$. 
\item[(b)] $\Psi_r$ is the \emph{$r$-th homological Adams operation}, which on $H^{BM}_{2k}(X^{(r)};\Q)$ ($k \in \Z$) is defined by multiplication by $\frac{1}{r^k}$, together with $y \mapsto y^r$.
\end{enumerate}
\end{theorem}

As special cases, we obtain for $f=id_X$ and, respectively, for the projection $f:X \times X' \to X$ the following:
\begin{cor}\label{cormain}
Let $X$ and $X'$ be  complex quasi-projective varieties. Then the following identities hold in the Pontrjagin ring $PH_*(X)$:
\begin{equation}\label{te100}
\sum_{n \geq 0} {{T_{(-y)}}_*} (X^{(n)}) \cdot t^n= \exp \left( \sum_{r \geq 1} \Psi_r  \left(d^r_* {{T_{(-y)}}_*}(X) \right) \cdot \frac{t^r}{r} \right),\end{equation}
and 
\begin{equation}\label{te101}
\sum_{n \geq 0} {{T_{(-y)}}_*} \Big([(X \times X')^{(n)} \to \Xs] \Big) \cdot t^n= \exp \left( \sum_{r \geq 1} \Psi_r  d^r_* \Big({\chi_{-y}(X')} \cdot  {{T_{(-y)}}_*}(X) \Big) \cdot \frac{t^r}{r} \right).\end{equation}
\end{cor}

\begin{remark}\rm
Formula (\ref{te101}) is one of the key ingredients for obtaining generating series formulae for the push-forwards (under the Hilbert-Chow morphisms) of the motivic Hirzebruch classes of {\it Hilbert schemes of points} of a smooth quasi-projective variety, see our recent preprint \cite{CMOSY} for details.
\end{remark}

To see that formula (\ref{te101})  is indeed a consequence of Theorem \ref{t0}, we note that  $$[f:X \times X' \to X]=[id_X] \boxtimes [X' \to pt],$$ so that 
$${{T_{(-y)}}_*}([f])={\chi_{-y}(X')} \cdot  {{T_{(-y)}}_*}(X).$$ This is a special case of the fact that the motivic Hirzebruch class  transformation $T_{y*}$ commutes with cross-products, see \cite{BSY}.  

Theorem \ref{t0} is a special case of  a more general result about a corresponding generating series for  un-normalized Hirzebruch classes of (complexes of) mixed Hodge modules, as we shall explain next. 
A mixed-Hodge-module-free proof of Theorem \ref{t0} will be given in Section \ref{s0}, by reducing it to the special case $f=id_M$ for $M$ a smooth quasi-projective variety.   The latter case is a corollary of 
Moonen's generating series formula for his generalized Todd classes $\tau_y(M^{(n)})$ of symmetric products of a projective manifold $M$ (see \cite{M}[Satz $2.1_y$, p.172]), together with the identification of $$\tau_y(M^{(n)})={T_y}_*(M^{(n)})$$ with the corresponding un-normalized Hirzebruch class.


\subsection{Hirzebruch classes of mixed Hodge modules}

As shown in \cite{BSY,Sch-MSRI}, the un-normalized motivic Hirzebruch class transformation factors through the Grothendieck group $K_0(\mh(X))$ of (algebraic) mixed Hodge modules on $X$ \cite{Sa}, by a transformation 
\begin{equation}\label{m} {T_y}_*:K_0(\mh(X)) \to H_{ev}^{BM}(X) \otimes \Q[y^{\pm 1}],\end{equation}
together with the natural group homomorphism
\begin{equation}
\chi_{\rm Hdg}: K_0(var/X) \to K_0(\mh(X)) \ , \ [f:Y \to X] \mapsto [f_! \Q^H_Y] \ ,
\end{equation}
for $\Q^H_Y$ the constant mixed Hodge module complex on $Y$.
So the un-normalized homology Hirzebruch class ${T_y}_*(X)$ can also be defined as ${T_y}_*([\Q^H_X])$.

One advantage of using mixed Hodge modules is that we can evaluate the transformation ${T_y}_*$ on other interesting ``coefficients''. For example,  a ``good"  (i.e., graded polarizable, admissible and with quasi-unipotent monodromy at infinity) variation $\LL$ of mixed Hodge structures on a smooth variety  $X$ yields twisted Hirzebruch classes ${T_y}_*(X,\LL)$,  the shifted intersection cohomology Hodge module $IC'^H_X:=IC^H_X[-\dim (X)]$ yields, for $X$ pure-dimensional, similar intersection Hirzebruch classes ${IT_y}_*(X):=T_{y*}([IC'^H_X])$, and  for a generically defined ``good"  variation $\LL$ of mixed Hodge structures on $X$ we obtain twisted intersection Hirzebruch classes ${IT_y}_*(X,\LL)$ associated to the shifted twisted intersection Hodge module $IC'^H_X(\LL):=IC^H_X(\LL)[-\dim (X)]$.

Over a point space $X=\{pt\}$, the transformation ${T_y}_*$ (as well as  its normalization defined in \S\ref{r1} below) reduces to the $\chi_y$-polynomial ring homomorphism 
\begin{equation}\chi_y: K_0(\ms^p) \to \Z[y,y^{-1}],\end{equation} which is defined on the Grothendieck group $K_0(\ms^p)$  of (graded) polarizable mixed Hodge structures by: \begin{equation}\label{def-chi}\chi_y([H]):=\sum_p {\rm dim}_{\C} Gr^p_F(H \otimes \C) \cdot (-y)^p,\end{equation} with $F^{\bullet}$ the Hodge filtration on $H \in \ms^p$. 
So,  if $X$ is a compact variety, by pushing down to a point the classes ${T_y}_*(X)$ and ${IT_y}_*(X)$ (or their normalized counterparts from \S \ref{r1}), one gets that the degrees of their zero-dimensional 
components are the corresponding Hodge polynomials $\chi_y(X)$ and $I\chi_y(X)$, respectively, defined in terms of dimensions of the graded parts of the Hodge filtration on the (intersection) cohomology of $X$. Similar considerations apply of course to the twisted versions ${T_y}_*(X,\LL)$ and resp. ${IT_y}_*(X,\LL)$ of these classes, and we denote their corresponding degrees by $\chi_y(X,\LL)$ and resp. $I\chi_y(X,\LL)$.

\bigskip

Before stating our main results on generating series of Hirzebruch classes in the mixed Hodge module context, we need to recall the definition of symmetric powers of (complexes of) mixed Hodge modules (see \cite{MSS} for more details).

Let $X$ be a complex quasi-projective variety, with $n$-th symmetric product $X^{(n)}$ and projection $\pi_n:X^n \to X^{(n)}$. For a complex of mixed Hodge modules $\MC \in D^b\mh(X)$, we let the {\it $n$-th symmetric power} of $\MC$ be defined by:
\begin{equation}\label{e}{\MC^{(n)}}:=({\pi_n}_*\MC^{\boxtimes{n}})^{\Sigma_n} 
\in D^b\mh(X^{(n)}),\end{equation} where  $\MC^{\boxtimes n} \in D^b\mh(X^n)$ is the $n$-th 
external product of $\MC$ with the $\Sigma_n$-action defined as in \cite{MSS}, and  $(-)^{\Sigma_n}$ is the projector on the $\Sigma_n$-invariant sub-object (which is well-defined since $D^b\mh(\Xs)$ is a Karoubian additive category).  The action of $\Sigma_n$ on $\MC^{\boxtimes{n}}$ is, by construction,  compatible with the natural action on the underlying $\Q$-complexes (see \cite{MSS} for details). As special cases of (\ref{e}), it was shown in \cite{MSS}[Rem.2.4(i)] that for $\MC=\Q_X^H$ the constant Hodge sheaf on $X$, one obtains:
\begin{equation}\label{q}\left(\Q_X^H\right)^{(n)}=\Q^H_{X^{(n)}}.\end{equation}
Moreover, it follows from the equivariant K\"unneth formula of \cite{MSS}[Sect.1.12] that for $\MC:=f_!\Q^H_Y$ with $f:Y \to X$ an algebraic map, we get \begin{equation}\label{qq}\left(f_!\Q_Y^H\right)^{(n)}=f^{(n)}_!(\Q^H_{Y^{(n)}}).\end{equation}
Also, for $\LL$ a polarizable variation of Hodge structure on a smooth
open subvariety $U$ of an irreducible variety $X$, we get 
\begin{equation}\label{tic}\left({IC'}_X^H(\LL)\right)^{(n)}={IC'}^H_{X^{(n)}}(\LL^{(n)}), \end{equation}
where the ``good" variation $\LL^{(n)}$ is generically defined on the regular part of $\Xs$ (see \cite{MSS}[Rem.2.4(ii)]). In particular, for the constant variation we have that
\begin{equation}\label{ic}\left({IC'}_X^H\right)^{(n)}={IC'}^H_{X^{(n)}}.\end{equation}

The main result of this paper is the following generating series formula for the Hirzebruch 
classes of the symmetric powers $\MC^{(n)} \in D^b\mh(X^{(n)})$ of a fixed complex of mixed Hodge modules on the variety $X$:

\begin{theorem}\label{t1}
Let $X$ be a complex quasi-projective variety and $\MC \in D^b\mh(X)$. Then  
the following identity holds in the Pontrjagin ring $PH_*(X)$:
\begin{equation}\label{te1}
\sum_{n \geq 0} {{T_{(-y)}}_*} (\MC^{(n)} ) \cdot t^n= \exp \left( \sum_{r \geq 1} \Psi_r  \Big(d^r_* {{T_{(-y)}}_*}(\MC) \Big) \cdot \frac{t^r}{r} \right),\end{equation}
with $d_r:X \to X^{({r})}$ the canonical diagonal mapping, and $\Psi_r$ the $r$-th homological Adams operation.
\end{theorem}

The proof of Theorem \ref{t1} makes use of the Atiyah-Singer classes of \cite{CMSS1}, combined with the Lefschetz--Riemann--Roch theorem \cite{BFQ,M}, which in the context of symmetric products is related to  the singular Adams--Riemann--Roch transformation for coherent sheaves (e.g., see \cite{FL,M,N}).

\bigskip

If $X$ is a projective variety, by pushing down to a point the result of Theorem \ref{t1}, we recover the generating series formula for the Hodge polynomials $\chi_y(X^{(n)} , \MC^{(n)}):=\chi_y([H^*(\Xs;\MC^{(n)}])$ (cf. \cite{MS09,MSS}), namely:
\begin{equation}
\sum_{n \geq 0} \chi_{-y}(X^{(n)} , \MC^{(n)}) \cdot t^n=\exp\left( \sum_{r \geq 1} \chi_{-y^r}(X,\MC) \cdot \frac{t^r}{r} \right).
\end{equation}
Indeed, over a point space, the map $d^r$ is the identity, and the $r$-th Adams operation $\Psi_r$  becomes $y \mapsto y^r$.

\bigskip

If we let $\MC$ be the constant Hodge sheaf $\Q_X^H$, we get back (by using (\ref{q})) our formula (\ref{te100}) for the motivic Hirzebruch classes. Similarly, for $\MC=f_!\Q^H_Y$ with $f:Y \to X$ an algebraic map, we get by (\ref{qq}) the generating series formula (\ref{te0}) of Theorem \ref{t0}.
Finally, for the shifted twisted intersection chain sheaf ${IC'}_X^H(\LL)$ we obtain by (\ref{tic}) the following special case of formula (\ref{te1}), as announced in \cite{MSL} for the case of the constant variation:
\begin{cor}\label{c1}
For $\LL$ a polarizable variation of Hodge structures on a smooth
open subvariety $U$ of an irreducible variety $X$,   
the following identity holds in $PH_*(X)$:
\begin{equation}\label{cl2}
\sum_{n \geq 0} {{IT_{(-y)}}_*} (X^{(n)},\LL^{(n)}) \cdot t^n= \exp \left( \sum_{r \geq 1} \Psi_r  \Big(d^r_* {{IT_{(-y)}}_*}(X,\LL) \Big) \cdot \frac{t^r}{r} \right).\end{equation} \end{cor}

If $X$ is smooth and projective, and $\MC$ is the constant Hodge sheaf $\Q_X^H$, formula (\ref{te1}) specializes to Moonen's generating series formula for his generalized Todd classes $\tau_y(X^{(n)})$ (cf. \cite[p.172]{M}). Indeed, as shown in equation (17) of \cite{CMSS1}, Moonen's generalized Todd class $\tau_y(Y/G)$, which he could only define for a projective orbifold $Y/G$ (with $G$ a finite group of algebraic automorphisms of the projective manifold $Y$), 
 coincides in this context with the Brasselet--Sch\"urmann--Yokura un-normalized Hirzebruch class ${T_y}_*(Y/G)$ considered in this paper.


\subsection{Todd, $L$- and Chern classes}

We conclude this introduction with a discussion on other important special cases of the main Theorem \ref{t1} and its Corollary \ref{c1}. 

\bigskip 

If $y=0$ and $\MC=\Q^H_X$, the formula for the corresponding classes ${T_0}_*(-)$ and  ${IT_0}_*(-)$  should be compared with Moonen's generating series formula for the  Todd classes $td_*(X^{(n)})$ of symmetric products of a projective variety (see \cite{M}[p.162--164]). However, while these three classes satisfy the same generating series formula, they do not coincide in general, except in very special cases, e.g., ${T_0}_*(X)=td_*(X)$ if $X$ has only Du Bois singularities (e.g., $X$ is smooth or has only rational singularities) so that the symmetric products $X^{(n)}$ also have only  Du Bois singularities (see \cite{KS}[Cor.5.4]). If $X$ is projective with only Du Bois singularities, by taking the degrees of the zero-dimensional components in the corresponding generating series formula, we recover Moonen's generating series formula for the arithmetic genus $\chi_a$ of symmetric products of such a projective variety (cf. \cite{M}[Corollary 2.7, p.161-162]):
\begin{equation} \sum_{n \geq 0} \chi_a(X^{(n)}) t^n = \exp \left( \sum_{r \geq 1} \chi_a(X) \cdot \frac{t^r}{r} \right)= (1-t)^{-\chi_a(X)} , 
\end{equation}
where $\chi_a(X):=\chi([H^*(X;\OO_X)]).$

\bigskip 

Let us now consider the case $y=-1$, and assume that $X$ is an irreducible projective variety of even complex dimension, with $\LL$ a ``good" polarizable variation of Hodge structures of even weight generically defined on $X$. Then by taking the degree of the zero-dimensional components in (\ref{cl2}), we recover the generating series formula of \cite{MS09,MSS} for the twisted Goresky--MacPherson intersection cohomology signature $\sigma(X^{(n)},\LL^{(n)})$ of the symmetric products of $X$,  i.e., 
\begin{equation}\label{Za-sig}
\sum_{n \geq 0} {\sigma}(X^{(n)},\LL^{(n)}) \cdot t^n=\frac{(1+t)^{\frac{\sigma(X,\LL)-I\chi(X,\LL)}{2}}}{(1-t)^{\frac{\sigma(X,\LL)+I\chi(X,\LL)}{2}}},
\end{equation} 
with $I\chi(X,\LL):=\chi([IH^*(X;\LL)])$ the twisted intersection cohomology Euler characteristic of $X$. Here, the twisted intersection cohomology $IH^*(X;\LL)$ gets an induced {\it symmetric} intersection pairing in middle degree since $\LL$ is of even weight, so that the corresponding signature $\sigma(X,\LL)$ is defined.   Similarly for the twisted signatures $ {\sigma}(X^{(n)},\LL^{(n)})$ of symmetric products. Since $X$ is of even complex dimension, the sheaf complex $IC_X(\LL)$ also has a {\it symmetric} self-duality structure yielding the same signature as $\sigma(X,\LL)$.
Moreover, under our assumptions, the pure Hodge module $IC^H_X(\LL)$ is also of even weight, so that we can use 
Saito's Hodge index theorem (e.g., see \cite{MSS}[Sect.3.6]) for the following identification of the twisted intersection cohomology signature in terms of the Hodge polynomial:
$$\sigma(X,\LL)=I\chi_1(X,\LL)=\chi_y([IH^*(X;\LL)])|_{y=1}.$$
If, moreover, $X$ is smooth and $\LL$ is the constant variation,  formula (\ref{Za-sig}) was proved by Zagier \cite{Za}. With regard to characteristic classes,  formula (\ref{cl2}) specialize for $X$ smooth and projective to the generating series for Moonen's class ${\tau_1}(X^{(n)})$ of symmetric products of $X$ (see \cite{M}[Cor.2.13, p.173]). This class differs from the Thom--Milnor homology $L$-class $L_*(X^{(n)})$ by a renormalization, defined by multiplying in each even degree by a suitable power of $2$. More precisely, for any projective $G$-manifold $Y$, with $G$ a finite group of algebraic automorphisms of $Y$, one has (cf. \cite{M}[Corollary 2.10, p.171]): $$\Psi_2 {T_1}_*(Y/G)=\Psi_2 {\tau_1}(Y/G)=L_*(Y/G),$$ with $\Psi_2$ the second homological Adams operation (as defined in Theorem \ref{t0}). A formula for the Thom--Milnor $L$-classes of symmetric products was originally obtained by Zagier \cite{Za} in the manifold context, and then re-proved by Moonen in \cite{M} in the case of complex projective manifolds. For the specialization of our formula (\ref{cl2}) to $y=-1$, we need to identify the twisted classes ${IT_{-1}}_*(X,\LL)$ as well as  ${IT_{1}}_*(\Xs,\LL^{(n)})$. 
By \cite{Sch-MSRI}[Prop.5.21], we get in degree $2k$ that $${IT_{-1,k}}(X,\LL)=(1+y)^kc_k(IC'_{X}(\LL))|_{y=-1} \in H_{2k}^{BM}(\Xs;\Q)$$
(the reader should be aware that in loc. cit. the identification is in terms of  normalized Hirzebruch classes). Therefore, $${IT_{-1*}}(X,\LL)=c_0(IC'_{X}(\LL))=I\chi(X,\LL)$$ is the degree-zero MacPherson Chern class of the underlying twisted intersection cohomology complex, i.e., the corresponding twisted Euler characteristic under the identification $deg: H_0^{BM}(X)=H_0(X) \simeq \Q$. Conjecturally, 
$$\Psi_2{IT_{1}}_*(\Xs,\LL^{(n)})=L_*(\Xs,\LL^{(n)}):=L_*(IC_{\Xs}(\LL^{(n)})),$$ with $L_*(IC_{\Xs}(\LL^{(n)}))$ the Cappell-Shaneson homology $L$-class \cite{CS} of the underlying self-dual constructible sheaf complex. At least for $X$ smooth and $\LL$ defined now on all of $X$, this identification follows from equation (23) in \cite{CMSS1} because then $\LL^{(n)}$ extends as a local system on all of $\Xs$. In this case, formula (\ref{cl2}) specializes for $y=-1$ and composition with $\Psi_2$ to the following twisted version of Moonen's $L$-class formula mentioned above (the reader should be aware that the exponent $\frac{1}{2}$ is missing from loc. cit.):
\begin{equation}\label{tl}
\sum_{n \geq 0} L_* (X^{(n)},\LL^{(n)}) \cdot t^n= (1-t^2)^{-\frac{I\chi(X,\LL)}{2}} \cdot 
\exp \left( \sum_{r \geq 1} \Psi_{2r-1}  \Big(d^{2r-1}_* L_*(X,\LL) \Big) \cdot \frac{t^{2r-1}}{2r-1} \right).
\end{equation}

\bigskip

For $y=1$, $X$ projective and $\MC=\Q^H_X$, by taking degrees in formula (\ref{te1}) we recover's MacDonald's generating series formula for the Euler characteristics of symmetric products \cite{Mac}: \begin{equation} \sum_{n \geq 0} \chi(X^{(n)}) t^n = \exp \left( \sum_{r \geq 1} \chi(X) \cdot \frac{t^r}{r} \right)= (1-t)^{-\chi(X)} . 
\end{equation}
Indeed, by (\ref{def-chi}), we see that $\chi_{-1}$ is just the usual Euler characteristic. 
Similarly, by taking degrees in formula (\ref{cl2}) for the constant variation, we obtain the generating series formula for the intersection cohomology Euler characteristic $I\chi(X^{(n)})$ of the symmetric products of $X$, see \cite{MS09,MSS}. Finally, after a suitable re-normalization (as explained in \S \ref{r1}), formula (\ref{te1}) specializes for the value $y=1$ of the parameter and $\MC=\Q^H_X$ the constant Hodge module complex  to Ohmoto's generating series formula \cite{Oh} for the rationalized MacPherson--Chern classes $c_*(X^{(n)})$ of the symmetric products of $X$ (see \S \ref{r1} for details):
\begin{equation}\label{Oh}
\sum_{n \geq 0} {c_*} (X^{(n)} ) \cdot t^n= \exp \left( \sum_{r \geq 1}  d^r_* {c_*}(X)  \cdot \frac{t^r}{r} \right) .
\end{equation}

Adapting our method of proving the main Theorem \ref{t1}, we get similarly the following counterparts for bounded sheaf complexes with constructible  resp. coherent cohomology sheaves:
\begin{theorem}\label{t2} Let $X$ be a complex quasi-projective variety. Then the following formulae hold in the Pontrjagin ring $PH_*(X)$ with $\Q$-coefficients:
\begin{itemize}
\item[(a)] For $\FC \in D^b_c(X)$ a constructible sheaf complex, we have:
\begin{equation}\label{sh2}
\sum_{n \geq 0} {c_*} (\FC^{(n)} ) \cdot t^n= \exp \left( \sum_{r \geq 1}  d^r_* {c_*}(\FC)  \cdot \frac{t^r}{r} \right) ,
\end{equation}
\item[(b)] For $\GC \in D^b_{coh}(X)$ a sheaf complex of $\OO_X$-modules with coherent cohomology, we have:
\begin{equation}\label{sh3}
\sum_{n \geq 0} {td_*} (\GC^{(n)} ) \cdot t^n= \exp \left( \sum_{r \geq 1}  \Psi_r \Big( d^r_* {td_*}(\GC)  \Big)\cdot \frac{t^r}{r} \right).
\end{equation}
\end{itemize}
\end{theorem}
Here the corresponding symmetric powers of sheaf complexes are defined as in the mixed Hodge module context \cite{MSS}. 

If the constructible sheaf complex $\FC$ underlies a complex of mixed Hodge modules $\MC$, then (\ref{sh2}) is just a corollary of our main Theorem \ref{t1}. In particular, for $\FC=\Q_X$ the constant sheaf, we get back formula (\ref{Oh}). Also, for $X$ projective and $\GC=\OO_X$ the structure sheaf on $X$, formula (\ref{sh3}) specializes to Moonen's generating series formula for Todd classes of symmetric products. In fact, Moonen's method of proof for $\OO_X$ also yields this more general result for arbitrary coherent coefficients on a quasi-projective variety.
Finally, in the projective case, taking the degrees in formulae (\ref{sh2}) and (\ref{sh3}) yields the generating series formulae for the corresponding Euler characteristics, as already obtained in \cite{MS09}[Thm.1.4].

Our method of proof for the generating series formulae follows closely the strategy developed by Hirzebruch-Zagier \cite{Za} and Moonen \cite{M}. Moreover, it relies on the use of localized characteristic class transformations for {\it singular} spaces, e.g., the Atiyah-Singer class transformation of \cite{CMSS1}, the Lefschetz-Riemann-Roch transformation of \cite{BFQ, M}, or the localized Chern class transformation of \cite{Sch10}. A corresponding localized $L$-class transformation is not available, so we derive our $L$-class formula from the corresponding one for Hirzebruch classes of intersection cohomology complexes.

\begin{ack} We thank Morihiko Saito for our joint work, and for discussions about symmetric powers of mixed Hodge modules. 

S. Cappell and J. Shaneson are partially supported by DARPA-25-74200-F6188. L. Maxim is partially supported by NSF-1005338 and by a research fellowship from the Max-Planck-Institut f\"ur Mathematik, Bonn. J. Sch\"urmann is supported by the
SFB 878 ``groups, geometry and actions". S. Yokura is partially supported by Grant-in-Aid for
Scientific Research (No. 24540085), the Ministry of Education, Culture, Sports, Science and Technology (MEXT), JAPAN.
\end{ack}


\section{Generating series for motivic Hirzebruch classes}\label{s0}
In this section, we provide a mixed-Hodge-module-free proof of Theorem \ref{t0} from the Introduction.
As shown in \cite{BSY}, the motivic Hirzebruch class transformation 
$T_{y*}$ is functorial for proper push-forwards and it commutes with cross-products, so that $T_{y*}$ also commutes with the Pontrjagin product $\odot$.  We denote by the same symbol, $T_{y*}(-)$, the induced functorial ring homomorphism 
$$T_{y*}(-): (PK_0(var/X), \odot) \to (PH_*(X), \odot).$$
By using the decomposition $$(Z \cup U)^{(n)}\simeq \cup_{i+j=n} (Z^{(i)} \times U^{(j)})$$ we get a well-defined (semi-)group homomorphism (commuting with push-forwards)
$$\zeta:(K_0(var/X), +) \to (PK_0(var/X), \odot) \ , \ \ [f:Y\to X] \mapsto 1+ \sum_{n\geq 1} [f^{(n)}:Y^{(n)} \to \Xs] \cdot t^n$$
with values in the group of multiplicative units of the Pontrjagin ring $PK_0(var/X)$. The transformation $\zeta$ can be regarded as a relative version of the {\it Kapranov zeta function} \cite{K}, which one gets back by letting $X$ be  a point space. Therefore, formula (\ref{te0}) of Theorem \ref{t0} is equivalent to the commutativity of the following diagram of (semi-)group homomorphisms:
\begin{equation}\label{500}
\begin{CD}
(K_0(var/X), +) @>{T_{y*}}>> (H_{even}^{BM}(X)\otimes \Q[y],+)\\
@V{\zeta}VV  @VV{\exp\left( \sum_{r \geq 1} \Psi_r d^r_* (-) \frac{t^r}{r} \right)}V \\
(PK_0(var/X), \odot) @>{T_{y*}}>> (PH_*(X), \odot)  .
\end{CD}
\end{equation}
It is thus enough to check it on the generators of $(K_0(var/X), +)$ given by the classes $[f:M \to X]$ with $M$ smooth and quasi-projective and $f$ proper. By functoriality with respect to proper push-forwards of  (\ref{500}), it suffices in fact to assume $f=id_M$, with $M$ smooth and quasi-projective, i.e., to prove the formula
\begin{equation}\label{600}
\sum_{n \geq 0} {{T_{(-y)}}_*} (M^{(n)}) \cdot t^n= \exp \left( \sum_{r \geq 1} \Psi_r  \left(d^r_* {{T_{(-y)}}_*}(M) \right) \cdot \frac{t^r}{r} \right).\end{equation}
Since the Hirzebruch class transformation $ {{T_{(-y)}}_*} $ commutes with open pull-backs (see \cite{BSY}[Cor.3.1(iii)]), (by taking a smooth projective compactification of $M$) we can moreover assume that $M$ is smooth and projective. Formula (\ref{600}) in this case is exactly Moonen's 
generating series formula \cite{M}[Satz $2.1_y$, p.172] for his generalized Todd classes $\tau_y(M^{(n)})$ of symmetric products of a projective manifold $M$, once we have the identification $$\tau_y(M^{(n)})={T_y}_*(M^{(n)}).$$ 
A proof of the latter identity follows from the definition of the un-normalized homology Hirzebruch classes in terms of the filtered DuBois complex $(\underline{\Omega}^{\bullet}_X,F)$, namely (see \cite{BSY}):
$$T_{y*}(X):=\sum_{i,p \geq 0} (-1)^i td_*\left(\HC^i(gr^p_F(\underline{\Omega}^{\bullet}_X))\right) \cdot (-y)^p.$$
Indeed, for a (quasi-projective) orbifold $X=Y/G$ with $Y$ smooth and $G$ a finite group of algebraic automorphisms of $Y$, one gets by \cite{DB}[Thm.5.12] that $gr^p_F(\underline{\Omega}^{\bullet}_X)$ is cohomologically concentrated in degree $p$, with:
$$\HC^p(gr^p_F(\underline{\Omega}^{\bullet}_X))=(\pi_*(\underline{\Omega}^{p}_Y))^G,$$
for $\pi:Y \to Y/G$ the projection map. So by the definition of $\tau_y(X)$ in \cite{M}[p.167], we get:
$$T_{y*}(X)=\sum_{p \geq 0} td_*\left( (\pi_*(\underline{\Omega}^{p}_Y))^G \right) \cdot y^p=: \tau_y(X).$$\hfill$\square$
\begin{remark}\rm Note that Moonen uses a slightly different notion of a homology Pontrjagin ring $(H_{ev}(X^{(\infty)};R)[[t]], \bullet)$ (with $R=\Q$  or $\Q[y]$) for a {\it compact} topological space, see \cite{M}[Ch.II, Sect.2]. 
Here $X^{(\infty)}$ is defined as the limit of the directed system of all symmetric products $\Xs$, where the inclusion $\Xs \hookrightarrow X^{(n+1)}$ is well-defined only after choosing a basepoint in $X$. Since $X$ is compact, there is no difference between homology and Borel-Moore homology, and moreover, one gets injections $H_*(\Xs;R) \hookrightarrow H_*(X^{(\infty)};R)$. Thus our Pontrjagin ring $(PH_*(X), \odot)$ is just a subring of $(H_{ev}(X^{(\infty)};R)[[t]], \bullet)$. And by definition, all of the above generating series live already in this subring. Finally, note that Moonen's notion of Pontrjagin ring only applies to projective (hence compact) varieties, while the version used in this paper can also be used in the quasi-projective context.
\end{remark}


\section{Generating series for Hirzebruch, Todd and Chern classes}\label{s1} 
In this section, we prove our main Theorem \ref{t1} and Theorem \ref{t2}. 

\subsection{Hirzebruch classes of symmetric powers of mixed Hodge modules}\label{mhm} 
An essential ingredient in the proof of Theorem \ref{t1} is the Atiyah--Singer class transformation (cf. \cite{CMSS1}) $${T_y}_*(-; g):K_0(\mh^G(X)) \to H^{BM}_{ev}(X^g) \otimes \C[y^{\pm 1}],$$ which is defined by combining Saito's theory with the Lefschetz--Riemann--Roch transformation 
$$td_*(-; g):K_0(\co^G(X)) \to H^{BM}_{ev}(X^g;\C)$$
of Baum--Fulton--Quart \cite{BFQ} and Moonen \cite{M}. These transformations are defined for any complex quasi-projective variety $X$ acted upon by a finite group $G$ of algebraic automorphisms. Here $K_0(\mh^G(X))$ denotes the Grothendieck group of equivariant mixed Hodge modules, which is identified with a suitable Grothendieck group of ``weakly" equivariant complexes of mixed Hodge modules (see \cite[Appendix A]{CMSS1}). Also, $K_0(\co^G(X))$ denotes the Grothendieck group of $G$-equivariant algebraic coherent sheaves on $X$. More details on the construction of the Atiyah--Singer class transformation ${T_y}_*(-; g)$ will be given 
in \S \ref{S3}, as needed.

\bigskip

For the following notions and strategy of proof we follow Hirzebruch-Zagier \cite{Za}[Ch.II] and Moonen \cite{M}[Ch.II, Sect.2], which deal with $L$-classes of symmetric products of a rational homology manifold and, respectively,  Todd classes of symmetric products of a complex projective variety.

\bigskip 

Let $\sigma \in \Sigma_n$ have cycle partition $\lambda=(k_1, k_2, \cdots )$, i.e.,  $k_r$ is the number of length $r$ cycles in $\sigma$ and $n=\sum_r r \cdot k_r$. Let $$\pi^{\sigma}:(X^n)^{\sigma} \to X^{(n)}$$ denote the composition of the inclusion of the fixed point set $(X^n)^{\sigma} \hookrightarrow X^n$ followed by the projection $\pi_n:X^n \to X^{(n)}$.
For a cycle $A$ of $\sigma$, we let $\vert A \vert$ denote its length. Then 
$$(X^n)^{\sigma} \simeq \prod_{{\text A  = \, \, cycle \ in} \ \sigma} 
(X^{|A|})^A\simeq \prod_r \left( (X^r)^{\sigma_r} \right)^{k_r} \simeq \prod_r \Delta_r(X)^{k_r} \simeq X^{k_1+k_2+\cdots} \ ,$$ 
where $\sigma_r$ denotes a cycle of length $r$, and $\Delta_r(X)$ is the diagonal in $X^r$. Also,  
$(X^{|A|})^A \simeq X$, diagonally embedded in $X^{\vert A \vert}$.  Here the inclusion $X^{\vert A \vert} \hookrightarrow X^n$ is given by $X_{j_1} \times X_{j_2} \times \cdots $, for $A=(j_1,j_2, \cdots)$ and with $X_j$ on the $j$-th place in $X^n$. Then the projection $\pi^{\sigma}:(X^n)^{\sigma} \to X^{(n)}$ is the product (over cycles $A$ of $\sigma$) of projections $$\pi^{A}:X \to X^{(\vert A \vert)}$$ defined by the composition $$\pi^{A}:X \simeq \Delta_{\vert A \vert}(X) \hookrightarrow X^{\vert A \vert} \to X^{(\vert A \vert)},$$ with $$\prod_{A}  X^{(\vert A \vert)} \to X^{(n)}$$ induced by the Pontrjagin product. In the notations of Theorem \ref{t1}, this amounts to saying that $\pi^{\sigma}$ is the product of projections $$d^r:X \simeq \Delta_r(X) \overset{i_r}{\hookrightarrow} X^r \overset{\pi_r}{\to} X^{(r)},$$ where each $r$-cycle contributes a copy of $d^r$.

\bigskip

Theorem \ref{t1} is a consequence of the following sequence of reductions:

\begin{lemma}[Averaging property]\label{l1} For  $\MC \in D^b\mh(X)$ and every $n \geq 0$, we have:
\begin{equation}\label{le1}
{T_{y}}_*(\MC^{(n)}) =\frac{1}{n!}\sum_{\sigma \in \Sigma_n} \pi^{\sigma}_*{T_{y}}_*(\MC^{\boxtimes{n}}; \sigma).
\end{equation}
\end{lemma}
\begin{proof} This follows directly from \cite{CMSS1}[Theorem 5.4], by regarding the external product $\MC^{\boxtimes{n}}$ with its $\Sigma_n$-action (as defined in \cite{MSS}) as a weakly equivariant complex.

\end{proof}

\begin{lemma}[Multiplicativity]\label{l2} If $\sigma \in \Sigma_n$ has cycle-type $(k_1, k_2, \cdots)$, then:
\begin{equation}\label{le2}
{T_{y}}_*(\MC^{\boxtimes{n}}; \sigma)=\prod_r \left( {T_{y}}_*(\MC^{\boxtimes{r}}; \sigma_r) \right)^{k_r}.
\end{equation}
Therefore,
\begin{equation}\label{le3}
\pi^{\sigma}_*{T_{y}}_*(\MC^{\boxtimes{n}}; \sigma)=\prod_r \left( d^r_*{T_{y}}_*(\MC^{\boxtimes{r}}; \sigma_r) \right)^{k_r}
\end{equation}
\end{lemma}
\begin{proof} This is a consequence of the multiplicativity property of the Atiyah--Singer 
 class  transformation, see \cite{CMSS1}[Corollary 4.2]. 

\end{proof}

\begin{lemma}[Localization]\label{l3} The following identification holds in $H^{BM}_{ev}(X) \otimes \Q[y^{\pm 1}] \subset H^{BM}_{ev}(X) \otimes \C[y^{\pm 1}]$:
\begin{equation}\label{le4}
{T_{(-y)}}_*(\MC^{\boxtimes{r}}; \sigma_r)=\Psi_r {T_{(-y)}}_*(\MC),
\end{equation}
with $\Psi_r$ the $r$-th homological Adams operation, which is defined on $H^{BM}_{2k}(X;\Q)$ ($k \in \Z$) by multiplication by $\frac{1}{r^k}$, together with $y \mapsto y^r$.
\end{lemma}

The localization property of Lemma \ref{l3} is the key technical point, and also the most difficult one. 
Its proof will be given later on,  in  Section \ref{S3}, while in the remaining part of this section  we concentrate on generating series formulae.

\bigskip

We now have all the ingredients for proving Theorem \ref{t1}.

\begin{proof}  For a given {\it partition} $\Pi=(k_1,k_2,\cdots,k_n)$ of $n$, i.e., $n=\sum_r k_r \cdot r$, denote by $N_{\Pi}$ the number of elements $\sigma \in \Sigma_n$ of cycle-type $\Pi$. Then it's easy to see that $$N_{\Pi}=\frac{n!}{k_1!k_2! \cdots 1^{k_1}2^{k_2} \cdots}.$$
Formula (\ref{te1}) follows now  from the following sequence of identities:
\begin{eqnarray*} \sum_n {T_{(-y)}}_*(\MC^{(n)}) \cdot t^n
&\overset{(\ref{le1})}{=}& \sum_{n} t^n \cdot \frac{1}{n!}\sum_{\sigma \in \Sigma_n} \pi^{\sigma}_*{T_{(-y)}}_*(\MC^{\boxtimes{n}};\sigma) \\
&\overset{(\ref{le3})}{=}& \sum_n \frac{t^n}{n!} \cdot \sum_{\Pi=(k_1,k_2,\cdots k_n)}  N_{\Pi} \prod_{r=1}^n \left( d^r_*{T_{(-y)}}_*(\MC^{\boxtimes{r}};\sigma_r) \right)^{k_r}\\
&=& \sum_n \sum_{\Pi=(k_1,k_2,\cdots k_n)} \frac{t^{k_1 \cdot 1+k_2 \cdot 2 + \cdots}}{k_1!k_2! \cdots 1^{k_1}2^{k_2} \cdots} \prod_{r=1}^n \left( d^r_*{T_{(-y)}}_*(\MC^{\boxtimes{r}};\sigma_r) \right)^{k_r}\\
&=&  \sum_n \sum_{\Pi=(k_1,k_2,\cdots k_n)} \prod_{r=1}^n \frac{t^{{k_r}\cdot r}}{k_r! r^{k_r}} \cdot \left( d^r_*{T_{(-y)}}_*(\MC^{\boxtimes{r}};\sigma_r) \right)^{k_r} \\
&=& \prod_{r=1}^{\infty} \left(  \sum_{k_r=0}^{\infty} \frac{t^{k_r \cdot r}}{k_r!r^{k_r}} \cdot \left( d^r_*{T_{(-y)}}_*(\MC^{\boxtimes{r}};\sigma_r) \right)^{k_r} \right)\\
&=& \prod_{r=1}^{\infty} \left(  \sum_{k_r=0}^{\infty} \frac{1}{k_r!}  \cdot  \left( d^r_*{T_{(-y)}}_*(\MC^{\boxtimes{r}};\sigma_r) \cdot \frac{t^r}{r} \right)^{k_r} \right)\\
&=& \prod_{r=1}^{\infty} \exp \left( d^r_*{T_{(-y)}}_*(\MC^{\boxtimes{r}};\sigma_r) \cdot  \frac{t^r}{r}   \right)\\
&=& \exp \left( \sum_{r=1}^{\infty} d^r_*{T_{(-y)}}_*(\MC^{\boxtimes{r}};\sigma_r) \cdot \frac{t^r}{r} \right)\\
&\overset{(\ref{le4})}{=}& \exp \left( \sum_{r=1}^{\infty} \Psi_r \left( d^r_*{T_{(-y)}}_*(\MC) \right) \cdot \frac{t^r}{r} \right),
\end{eqnarray*} 
where in the last equality we also use the functoriality with respect to proper push-down of the homological Adams transformation $\Psi_r$.

\end{proof}

The proof of Theorem \ref{t2} proceeds formally in exactly the same way as the above proof, once the {\it averaging}, {\it multiplicativity} and {\it localization} properties hold for the corresponding Lefschetz-Riemann-Roch transformation $td_*(-;g)$  of Baum--Fulton--Quart \cite{BFQ} and Moonen \cite{M}, and respectively, for the Sch\"urmann's localized Chern classes $c_*(-;g)$ from \cite{Sch10}[Ex.1.3.2]. These will be explained in the next subsections.


\subsection{Lefschetz-Riemann-Roch transformation}\label{Todd}

Let $Z$ be a complex quasi-projective algebraic variety, and $G$ a finite group of algebraic automorphisms of $Z$. Let \begin{equation}
td_*(-; g):K_0(\co^G(Z)) \to H^{BM}_{ev}(Z^g;\C)
\end{equation}
be the Lefschetz--Riemann--Roch transformation of Baum--Fulton--Quart \cite{BFQ} and Moonen \cite{M}, where $K_0({\rm Coh}^G(Z))$ is the Grothendieck group of $G$-equivariant algebraic coherent sheaves on $Z$.
Its values yield homology classes localized on the fixed-point set $Z^g$.

For a complex $\E \in D^b_{coh}(Z)$ which is weakly $G$-equivariant in the sense of \cite{CMSS1}, its cohomology sheaves are $G$-equivariant coherent sheaves, so that we can define its class in the Grothendieck  group $K_0(\co^G(Z))$ as $$[\E]:=\sum\nolimits_{i} (-1)^i \cdot [\HC^i(\E)].$$
Let $\pi:Z \to Z/G$ be the quotient map. Then $\pi_*\E \in D^b_{coh}(Z/G)$ is a weakly  $G$-equivariant  complex on $Z/G$. We can define its $G$-invariant part $(\pi_*\E)^G$ by using the corresponding projector as in \cite{CMSS1}[p.20] since $ D^b_{coh}(Z)$ is a $\Q$-linear additive Karoubian category. The class $[(\pi_*\E)^G]\in K_0(\co(Z/G))$ in the  Grothendieck group is defined as above.
With the above notations, we now have the following
\begin{lemma}[Averaging property]\label{l100}
\begin{equation}\label{800} td_*\Big((\pi_*\E)^G\Big) =\frac{1}{|G|}\sum_{g \in G} \pi^g_*td_*(\E;g),\end{equation} with $\pi^g:Z^g \to Z/G$ the induced map.
\end{lemma}
\begin{proof} Since $\pi_*$ and $(-)^G$ are exact functors, it suffices to prove formula (\ref{800}) for $\E$ a $G$-equivariant coherent sheaf. For  a projective variety $Z$, this fact follows from \cite{M}[p.170]. 
For a quasi-projective variety $Z$, it still holds true as soon as the Lefschetz--Riemann--Roch transformation $td_*(-; g)$ is extended to this quasi-projective context as in \cite{M}[Ch.III, Sect.1]. 

\end{proof}

Let now $X$ be a quasi-projective variety with $\GC \in D^b_{coh}(X)$ given. 
Then we can view the external product $\GC^{\boxtimes}$ with its induced $\Sigma_n$-action (e.g., see \cite{MSS}) as a weakly equivariant $\Sigma_n$-complex, and the $n$-th symmetric power $\GC^{(n)}$ of $\GC$ is defined by 
$$\GC^{(n)}:=({\pi_n}_* \GC^{\boxtimes})^{\Sigma_n} \in D^b_{coh}(\Xs).$$
The needed {\it averaging} property for the Lefschetz-Riemann-Roch transformation, i.e.,  
\begin{equation}\label{700} td_*(\GC^{(n)}) =\frac{1}{n!}\sum_{\sigma \in \Sigma_n} \pi^{\sigma}_*td_*(\GC^{\boxtimes{n}}; \sigma)\end{equation}
follows now from Lemma \ref{l100}.

\bigskip

Let us now discuss the {\it multiplicativity} property for the Lefschetz-Riemann-Roch transformation:
\begin{lemma}[Multiplicativity]\label{l101}
Let $Z$, $Z'$ be complex quasi-projective algebraic varieties, and $G$ and $G'$ finite groups of algebraic automorphisms of $Z$ and resp. $Z'$. Then for $\E \in D^b_{coh}(Z)$ a weakly $G$-equivariant complex and $\E' \in D^b_{coh}(Z')$ a weakly $G'$-equivariant complex, one has for $g \in G$ and $g' \in G'$ the following identity:
\begin{equation}\label{mult-lrr}
td_*(\E \boxtimes \E';(g,g'))=td_*(\E;g) \boxtimes td_*(\E';g').
\end{equation}
\end{lemma}
\begin{proof}
The external product is an exact functor, thus we can assume that $\E$ and $\E'$ are coherent sheaves. 
Moreover,  we can assume that $G=G'$ since the Lefschetz-Riemann-Roch transformation only depends on the cyclic group generated by the corresponding group element. So the claim follows from \cite{M}[7.9, p.125].

\end{proof}

Back in the case of symmetric products, Lemma \ref{l101} implies that for a fixed $\GC \in D^b_{coh}(X)$ we have
\begin{equation}\label{mult_sym}
td_*(\GC^{\boxtimes{n}}; \sigma)=\prod_r \left( td_*(\GC^{\boxtimes{r}}; \sigma_r) \right)^{k_r},
\end{equation}
for $\sigma \in \Sigma_n$ of cycle-type $(k_1, k_2, \cdots)$.

\bigskip

Finally, we have the following {\it localization} property for  the Lefschetz-Riemann-Roch transformation:
\begin{lemma}\label{l33a}
Let $\sigma_r$  be an $r$-cycle. Then for any $\GC \in D^b_{\rm coh}(X)$, the following identity holds in $H^{BM}_{ev}(X;\Q)$:
\begin{equation}\label{e33a}
td_*(\GC^{\boxtimes r}; \sigma_r)=\Psi_r td_*(\GC)
\end{equation}
under the identification  $(X^r)^{\sigma_r} \simeq X$.
\end{lemma}

\begin{proof} For $\GC=\OO_X$, this fact is proved by Moonen  (see \cite{M}[Satz 2.4, p.162]). His proof in this special case  uses an embedding $i: X \hookrightarrow M$ into a smooth complex algebraic variety $M$, together with a bounded locally free resolution $\FC$ of $i_*\GC$. Then $i^r:X^r \hookrightarrow M^r$ is a $\Sigma_r$-equivariant embedding, with $\FC^{\boxtimes r}$ a $\Sigma_r$-equivariant locally free resolution of $(i_*\GC)^{\boxtimes r} \simeq i^r_*(\GC^{\boxtimes r})$. 
Starting with  $\GC \in D^b_{\rm coh}(X)$ instead of $\OO_X$, Moonen's calculation of $td_*(\GC^{\boxtimes r}; \sigma_r)$  applies directly to this more general context.

\end{proof}

\begin{remark}\rm
Note that the homological Adams operation appearing in Lemma \ref{l33a} is induced from the $K$-theoretic Adams operation \cite{A}
$$\Psi^r: K^0(M,M\setminus X) \to K^0(M,M\setminus X) \otimes \C$$ appearing in Moonen's proof of his localization formula \cite{M}[p.164] as 
$$\Psi^r([\FC]):=[\Delta_r^*(\FC^{\boxtimes r})](\sigma_r),$$
with $\Delta_r:M \to M^r$ the diagonal embedding.
Note that the $\langle\sigma_r\rangle$-equivariant vector bundle complex $\Delta_r^*(\FC^{\boxtimes r})$
is exact off $X$, i.e., it defines a class $$[\Delta_r^*(\FC^{\boxtimes r})] \in K_{
\langle\sigma_r\rangle}^0(M,M \setminus X) \simeq K^0(M,M \setminus X) \otimes R({
\langle\sigma_r\rangle}),$$ where $R({
\langle\sigma_r\rangle})$ is the complex representation ring of the cyclic group generated by $\sigma_r$. Then $$(\sigma_r):K_{
\langle\sigma_r\rangle}^0(M,M \setminus X) \simeq K^0(M,M \setminus X) \otimes R({
\langle\sigma_r\rangle}) \to K^0(M,M \setminus X) \otimes \C$$
is induced by taking the trace homomorphism $tr(-;\sigma_r):R({
\langle\sigma_r\rangle})  \to \C$ (see \cite{M}[p.67]). 

Compare also with \cite{N}[Sect.3] for the relation between the Adams- and Lefschetz-Riemann-Roch theorems in the algebraic geometric context.
\end{remark}



\subsection{Localized Chern classes}\label{Chern}

Let $Z$ be a complex quasi-projective algebraic variety, and $G$ a finite group of algebraic automorphisms of $Z$. Let \begin{equation}
\begin{CD} c_*(-;g):K_0(\cs^G(Z)) @>{tr(-|_{Z^g};g)}>> F(Z^g) \otimes \Q @>{c_* \otimes \Q}>> H^{BM}_{ev}(Z^g;\Q)\end{CD}
\end{equation}
be Sch\"urmann's localized Chern class transformation of \cite{Sch10}[Ex.1.3.2], with $\cs^G(-)$ the abelian category of  algebraically constructible $G$-equivariant sheaves of $\Q$-vector spaces and $F(-)$ the abelian group of $\Z$-valued algebraically constructible functions. Here $c_*$ denotes the Chern class transformation of MacPherson \cite{MP}, and ${tr(-|_{Z^g};g)}$ is the group homomorphism defined by $$[\FC] \mapsto \Big(x \mapsto tr(\FC_x;g)\Big)\in F(Z^g).$$ Note that for $x \in Z^g$, $g$ acts on the finite-dimensional stalk $\FC_x$ for a constructible $G$-equivariant sheaf $\FC$. 

For  the trivial group element $id \in G$, $c_*(\FC):=c_*([\FC];id)$ defines the MacPherson-Chern class of a constructible sheaf $\FC$ as the MacPherson Chern class of the constructible function given by the stalkwise Euler characteristic.

For a complex $\FC \in D^b_c(Z)$ which is weakly $G$-equivariant in the sense of \cite{CMSS1}, its cohomology sheaves are $G$-equivariant constructible sheaves, so that we can define its class in the Grothendieck  group $K_0(\cs^G(Z))$ as $$[\FC]:=\sum\nolimits_{i} (-1)^i \cdot [\HC^i(\FC)].$$
Let $\pi:Z \to Z/G$ be the quotient map. Then $\pi_*\FC \in D^b_c(Z/G)$ is a weakly  $G$-equivariant  complex on $Z/G$. We can define its $G$-invariant part $(\pi_*\FC)^G$ using the corresponding projector as in \cite{CMSS1}[p.20] since $ D^b_c(Z)$ is a $\Q$-linear additive Karoubian category. Similarly, the class $[(\pi_*\FC)^G]\in K_0(\cs(Z/G))$ in the  Grothendieck group is defined as above.
With the above notations, we now have the following
\begin{lemma}[Averaging property]\label{l100c}
\begin{equation}\label{800c} c_*((\pi_*\FC)^G) =\frac{1}{|G|}\sum_{g \in G} \pi^g_*c_*(\FC;g),\end{equation} with $\pi^g:Z^g \to Z/G$ the induced map.
\end{lemma}
\begin{proof} Since $\pi_*$ and $(-)^G$ are exact functors, it suffices to prove formula (\ref{800c}) for a $G$-equivariant constructible sheaf. By functoriality of $c_*(-;g)$, this case follows from the constructible function identity:
\begin{equation}\label{810c}tr((\pi_*\FC)^G;id) =\frac{1}{|G|}\sum_{g \in G} \pi^g_* tr(\FC;g).\end{equation}
Pointwise, at $\pi(x) \in Z/G$, the left-hand side of the above equality reduces to the Euler characteristic of the following vector space:
$$((\pi_*\FC)_{\pi(x)})^G \simeq ( {\rm Ind}^G_{G_x} \FC_x)^G \simeq (\FC_x)^{G_x}, $$
with $G_x$ the stabilizer of $x \in Z$. The right-hand side of (\ref{810c}) can be calculated by the $G$-equivariance of $\FC$ as:
\begin{align*} \Big(\frac{1}{|G|}\sum_{g \in G} \pi^g_* tr(\FC;g)\Big)(\pi(x)) &= \frac{1}{|G|} \sum_{x',  \pi(x')=\pi(x)} \Big(\sum_{g \in G_{x'}} tr(\FC_{x'};g)\Big) \\
&= \frac{|G/{G_x}|}{|G|}  \Big(\sum_{g \in G_x}  tr(\FC_{x};g)\Big). \end{align*}
So the identity (\ref{810c}) reduces pointwise to the well-known identity for the rational $G_x$-representation $\FC_x$:
$$\chi(\FC_x^{G_x})= \frac{1}{|G_x|}\sum_{g \in G_x} tr(\FC_x;g).$$
\end{proof}

Let now $X$ be a quasi-projective variety with $\FC \in D^b_c(X)$ given. 
The $n$-th symmetric power $\FC^{(n)}$ of $\FC$ is defined as before by 
$$\FC^{(n)}:=({\pi_n}_* \FC^{\boxtimes})^{\Sigma_n} \in D^b_c(\Xs).$$
The needed {\it averaging} property for the localized Chern class transformation, i.e.,  
\begin{equation}\label{700c} c_*(\FC^{(n)}) =\frac{1}{n!}\sum_{\sigma \in \Sigma_n} \pi^{\sigma}_*c_*(\FC^{\boxtimes{n}}; \sigma)\end{equation}
follows now from Lemma \ref{l100c}.

\bigskip

Similarly, the {\it multiplicativity} property 
\begin{equation}\label{mult_c}
c_*(\FC^{\boxtimes{n}}; \sigma)=\prod_r \left( c_*(\FC^{\boxtimes{r}}; \sigma_r) \right)^{k_r}.
\end{equation}
for $\sigma \in \Sigma_n$ of cycle-type $(k_1, k_2, \cdots)$ follows from the following 
\begin{lemma}[Multiplicativity]\label{l101c}
Let $Z$, $Z'$ be complex quasi-projective algebraic varieties, and $G$ and $G'$ finite groups of algebraic automorphisms of $Z$ and resp. $Z'$. Then for $\FC \in D^b_c(Z)$ a weakly $G$-equivariant complex and $\FC' \in D^b_c(Z')$ a weakly $G'$-equivariant complex, one has for $g \in G$ and $g' \in G'$ the following identity:
\begin{equation}\label{mult-lrr}
c_*(\FC \boxtimes \FC';(g,g'))=c_*(\FC;g) \boxtimes c_*(\FC';g').
\end{equation}
\end{lemma}
\begin{proof}
The external product is an exact functor, thus we can assume that $\FC$ and $\FC'$ are constructible sheaves. 
Moreover,  we can assume that $G=G'$ since the localized Chern class  transformation only depends on the cyclic group generated by the corresponding group element. So the claim follows by the multiplicativity of the Chern class transformation (see \cite{BSY}) and multiplicativity of traces.

\end{proof}

\bigskip

Finally, we have the following {\it localization} property for  the localized Chern class transformation:
\begin{lemma}\label{l200c}
Let $\sigma_r$  be an $r$-cycle. Then for any $\FC \in D^b_c(X)$, the following identity holds in $H^{BM}_{ev}(X;\Q)$:
\begin{equation}\label{e33a}
c_*(\FC^{\boxtimes r}; \sigma_r)= c_*(\FC)
\end{equation}
under the identification  $(X^r)^{\sigma_r} \simeq X$.
\end{lemma}

\begin{proof} It suffices to show the constructible function identity:
$$tr(\FC^{\boxtimes r}; \sigma_r)= tr(\FC;id).$$
This can be checked pointwise, so we can assume that $X$ is a point space.
Choosing a free resolution, we can moreover assume that $\FC$ is a bounded complex of finite dimensional rational vector spaces. By Atiyah's definition of the Adams operation (e.g., see \cite{A}[Sect.2] and \cite{MS09}[Sect.3]),  one gets $$tr(\FC^{\boxtimes r}; \sigma_r)=\chi(\Psi^r([\FC])),$$ with $\Psi^r$ the $r$-th Adams operation acting on the Grothendieck group $K_0(\vct)$ of finite dimensional rational vector spaces. But  the Euler characteristic $\chi: K_0(\vct) \to \Z$ is a $\lambda$-ring isomorphism, hence $\Psi^r$ acts in fact trivially on $K_0(\vct)$. This yields
$$tr(\FC^{\boxtimes r}; \sigma_r)= \chi(\FC).$$

\end{proof}

Altogether, this finishes the proof of our Theorem \ref{t2}.


\section{Localization of Hirzebruch classes}\label{S3}
The aim of this section is to supply a proof of the technical Lemma \ref{l3}. We begin by recalling in \S \ref{S31} the construction of the Atiyah--Singer class transformation from \cite{CMSS1}. In \S \ref{S32}, we specialize to the case of symmetric products 
and study how Saito's functors $gr^F_*DR$ behave with respect to external powers. Then we finish 
the proof of Lemma \ref{l3}.
\subsection{The Atiyah--Singer class transformation}\label{S31}
Let $Z$ be a (possibly singular) quasi-projective variety acted upon by a finite group $G$ of algebraic automorphisms.
The Atiyah--Singer class transformation $${T_y}_*(-; g):K_0(\mh^G(Z)) \to H^{BM}_{ev}(Z^g) \otimes \C[y^{\pm 1}]$$  is constructed in \cite{CMSS1} in two stages. First, by using  Saito's theory of algebraic mixed Hodge modules  \cite{Sa}, we construct an equivariant version of the motivic Chern class transformation of \cite{BSY} (see also \cite{Sch-MSRI, Yokura-MSRI}), i.e., the {\em equivariant motivic Chern class transformation}:
\begin{equation}
\mc^G:K_0(\mh^G(Z)) \to K_0({\rm Coh}^G(Z)) \otimes \Z[y^{\pm 1}],
\end{equation}
for $K_0({\rm Coh}^G(Z))$ the Grothendieck group of $G$-equivariant algebraic coherent sheaves on $Z$.
Secondly, we employ  the Lefschetz--Riemann--Roch transformation of Baum--Fulton--Quart \cite{BFQ} and Moonen \cite{M}:
\begin{equation}
td_*(-; g):K_0(\co^G(Z)) \to H^{BM}_{ev}(Z^g;\C)
\end{equation}
to obtain (localized) homology classes on the fixed-point set $Z^g$. 

\bigskip

In order to define the equivariant motivic Chern class transformation $\mc^G$, we work in the category $D^{b,G} \mh(Z)$ of $G$-equivariant objects in the derived category $D^b\mh(Z)$ of algebraic mixed Hodge modules on $Z$, and similarly for $D_{\rm coh}^{b,G}(Z)$,  the category of $G$-equivariant objects in the derived category $D_{\rm coh}^{b}(Z)$ of bounded complexes of $\OO_Z$-sheaves with coherent cohomology. Let us recall that in both cases,  a $G$-equivariant element $\MC$ is just an element in the underlying additive category (e.g., $D^b\mh(Z)$), with a $G$-action given by isomorphisms
$$\psi_g: \MC \to g_* \MC \quad (g\in G),$$  
such that $\psi_{id}=id$ and $\psi_{gh}=g_*(\psi_h)\circ \psi_g$ for all $g,h \in G$ (see 
\cite{MS09}[Appendix]). These ``weak equivariant derived categories'' $D^{b,G}(-)$ are not triangulated in general. Nevertheless, one can define a suitable Grothendieck group, by using ``equivariant distinguished triangles'' in the underlying derived category $D^b(-)$, and get  isomorphisms (cf. \cite{CMSS1}[Lemma 6.7]): 
$$ K_0(D^{b,G} \mh(Z))=K_0(\mh^G(Z)) \quad \text{and} \quad K_0(D_{\rm coh}^{b,G}(Z))=
K_0({\rm Coh}^G(Z)).$$
The equivariant motivic Chern class transformation $\mc^G$ is defined by noting that Saito's natural transformations of triangulated categories (cf. \cite{Sa1}[Sect.2.2], \cite{Sa}[proof of Prop.2.33], \cite{Sa3}[(1.3.4)]) 
$${\rm gr}^F_p DR:D^b\mh(Z) \to D^b_{\rm coh}(Z)$$ commute with the push-forward $g_*$ induced by each $g \in G$, thus inducing equivariant transformations (cf. \cite{CMSS1}[Example 6.6]) $${\rm gr}^F_p DR^G:D^{b,G}\mh(Z) \to D^{b,G}_{\rm coh}(Z).$$ Note that for a fixed $\MC \in D^{b,G} \mh(Z)$, one has that ${\rm gr}^F_p DR^G(\MC)=0$ for all but finitely many $p \in \Z$.  This yields the following definition (cf. \cite{CMSS1}):
\begin{defn} The $G$-equivariant motivic Chern class transformation 
$$\mc^G:K_0(\mh^G(Z)) \to K_0(D_{\rm coh}^{b,G}(Z)) \otimes \Z[y^{\pm 1}] = 
K_0({\rm Coh}^G(Z)) \otimes \Z[y^{\pm 1}]$$
is defined by:
\begin{equation}\label{d}
\mc^G([\MC]):= \sum_{p} \left[{\rm gr}^F_{-p} DR^G (\MC) \right] \cdot (-y)^p=  \sum_{i,p}(-1)^i \left[\HC^i({\rm gr}^F_{-p} DR^G (\MC))\right] \cdot (-y)^p .
\end{equation} 
The Atiyah--Singer class transformation is defined by the composition
\begin{equation}
{T_y}_*(-; g):=td_*(-; g) \circ \mc^G,
\end{equation}
with 
\begin{equation}
td_*(-; g):K_0(\co^G(Z)) \to H^{BM}_{ev}(Z^g;\C)
\end{equation}
the Lefschetz--Riemann--Roch transformation (extended linearly over $ \Z[y^{\pm 1}]$).
\end{defn}


\subsection{Atiyah-Singer classes for external products}\label{S32}
In this section we finish the proof of Lemma \ref{l3}, after first developing the needed prerequisites.

\begin{lemma}\label{l32a} Let $X$ be a complex quasi-projective variety and fix $\MC \in D^b\mh(X)$. Then there is a $\Sigma_r$-equivariant isomorphism of bounded graded objects in $D^{b}_{\rm coh}(X^r)$:
\begin{equation}\label{gr}
gr^F_*DR(\MC^{\boxtimes r}) \simeq \left( gr^F_*DR(\MC) \right)^{\boxtimes r}, 
\end{equation}
where the left-hand side underlies the weakly equivariant complex $gr^F_*DR^{\Sigma_r}(\MC^{\boxtimes r})$, and the $\Sigma_r$-action on the right-hand side is the usual action on external products of graded complexes.
\end{lemma}

\begin{proof}
Since $X$ is quasi-projective, we can assume $X$ is embedded in a smooth complex algebraic variety $M$. We have $D^b\mh(X) \simeq D^b\mh_X(M)$, by using the identification of the category $\mh(X)$ of mixed Hodge modules on $X$ with the category  $\mh_X(M)$ of mixed Hodge modules on $M$ supported on $X$ (\cite{Sa}[Sect.4]). In this case, Saito's functor $gr^F_*DR$ is induced from the  graded transformation associated to a filtered de Rham functor
$$DR: D^b \mh_X(M) \to D^b_{{\rm coh}}F(\OO_M, {\rm Diff})$$ taking values in Saito's category of bounded filtered differential complexes on $M$ whose graded pieces are bounded complexes of  $\OO_M$-modules with coherent cohomology sheaves (\cite{Sa1}[\S 2.2]). Moreover, this filtered de Rham functor is induced from a corresponding functor of complexes 
$$DR: C^b \mh(M) \to C^bF(\DC_M) \to C^bF(\OO_M, {\rm Diff}),$$
associating to a complex of mixed Hodge modules on $M$ the filtered de Rham complex of the underlying complex of filtered right $\DC_M$-modules. Note that for $\MC \in C^b \mh_X(M)$, one has by \cite{Sa1}[Lem.3.2.6] that $gr^F_pDR(\MC)\in D^b_{\rm coh}(X)$ for all $p$  (compare also with \cite{Sa}[proof of Prop.2.33]), 
with $gr^F_pDR(\MC)\simeq 0$ for all but finitely many $p \in \Z$ (see \cite{Sa4}[Lem.1.14]). 

By \cite{MSS}[Remark 1.6], for $\MC \in C^b \mh(M)$ there is a canonical map 
$${\rm can} : DR(\MC^{\boxtimes r}) \to DR(\MC)^{\boxtimes{r}}$$ commuting with the corresponding $\Sigma_r$-actions as defined in \cite{MSS}. This induces a $\Sigma_r$-equivariant map $${\rm gr(can)}: gr^F_*DR(\MC^{\boxtimes r}) \to \left( gr^F_*DR(\MC) \right)^{\boxtimes{r}}$$  of the associated graded complexes. Moreover, ${\rm gr(can)}$ is a (graded) quasi-isomorphism, as can be checked locally using a suitable ``locally free" resolution as in \cite{Sa1}[Lemma 2.1.17].

Finally, by the multiple K\"unneth formula for push-forwards of mixed Hodge modules (\cite{MSS}[Sect.1.11]), the induced $\Sigma_r$-equivariant isomorphism (\ref{gr}) does not depend on the choice of the embedding.

\end{proof}

The above result together with Lemma \ref{l33a} yield the following:
\begin{prop} Let $X$ be a complex quasi-projective variety and fix $\MC \in D^b\mh(X)$. 
Then the following identity holds for  $\sigma_r$ an $r$-cycle:
\begin{equation} \label{imp}
td_*([gr^F_pDR^{\Sigma_r}(\MC^{\boxtimes r})] ;\sigma_r)=
\begin{cases} 
\Psi_r td_*([gr^F_{q}DR(\MC)]) \ , & \text{if  $p=q \cdot r$} \ , \\
0 \ , & \text{if  $p \not\equiv 0$  mod  r}. 
\end{cases} 
\end{equation}
\end{prop}

\begin{proof}
By taking the degree $p$ part in (\ref{gr}), we have that:
\begin{equation}\label{sum} 
gr^F_pDR^{\Sigma_r}(\MC^{\boxtimes r}) = \bigoplus_{\sum_{j=1}^r q_j=p} gr^F_{q_1}DR(\MC) \boxtimes \cdots \boxtimes  gr^F_{q_r}DR(\MC), 
\end{equation}
where the action of the $r$-cycle $\sigma_r$ on the right-hand side is the (graded anti-symmetric) action by cyclic permutations of the factors in the multiple external product of complexes (as explained e.g.,  in \cite{MSS}). 

Fix a multi-index $(q_1, \cdots , q_r) \in \Z^r$, with $\sum_{j=1}^r q_j=p$. If $q_1=\cdots=q_r=q$ (with $p=q \cdot r$),  we get by Lemma \ref{l33a} that
\begin{equation} 
td_*( [gr^F_{q}DR(\MC)^{\boxtimes r}];\sigma_r)=\Psi_r td_*([gr^F_{q}DR(\MC)]).
\end{equation}
Otherwise,  the orbit of $(q_1, \cdots , q_r)$ under the permutation action of $\sigma_r$ on $\Z^r$ has length $r$. This implies:
$$td_*\Big(\Big[\bigoplus_{j=1}^r  \ \GC_{\sigma_r^j(q_1)} \boxtimes \cdots \boxtimes  \GC_{\sigma_r^j(q_r)}\Big] ;\sigma_r \Big)=0 \ ,$$ for $\GC_q:=gr^F_{q}DR(\MC)$, $q \in \Z$.
This can be seen as follows: we first choose 
an embedding $i: X \hookrightarrow M$ into a smooth complex algebraic variety $M$, together with a bounded locally free resolution $\FC_q$ of $i_*\GC_q$  ($q \in \Z$). Then $$\bigoplus_{j=1}^r  \ \FC_{\sigma_r^j(q_1)} \boxtimes \cdots \boxtimes  \FC_{\sigma_r^j(q_r)}$$ is a $\sigma_r$-equivariant locally free resolution of 
$$i^r_*\Big(\bigoplus_{j=1}^r  \ \GC_{\sigma_r^j(q_1)} \boxtimes \cdots \boxtimes  \GC_{\sigma_r^j(q_r)}\Big) \ .$$ 
Let $\Delta_r:M \to M^r$ denote the diagonal embedding, with $\Delta_r(M) \simeq (M^r)^{\sigma_r}$. Then 
\begin{equation}\label{150}\Delta_r^* \Big(  \bigoplus_{j=1}^r  \ \FC_{\sigma_r^j(q_1)} \boxtimes \cdots \boxtimes  \FC_{\sigma_r^j(q_r)} \Big)\end{equation} is a complex, whose components are direct sums of terms of the form 
$$\left(  \FC^{k_1}_{q_1} \otimes \cdots \otimes  \FC^{k_r}_{q_r} \right) \otimes \left( \oplus_{j=1}^r  \OO_M \right) ,$$ for $ \FC_q^{k}$ the $k$-th degree component of the complex $\FC_q$, and with $\sigma_r$ acting (up to suitable signs) by cyclic permutation of order $r$ on the summands in $\oplus_{j=1}^r  \OO_M$. 

By using Alexander duality $H^*(M, M \setminus X;\C) \simeq H_*^{BM}(X;\C)$, the localized Todd class $td_*\Big(\Big[\bigoplus_{j=1}^r  \ \GC_{\sigma_r^j(q_1)} \boxtimes \cdots \boxtimes  \GC_{\sigma_r^j(q_r)}\Big] ;\sigma_r \Big)$ is defined as (see \cite{M}[Defn.4.5, p.72]):
$$td^*({T}{M}) \cup  \frac{ch_X^* \left(\Big[\Delta_r^* \left(  \bigoplus_{j=1}^r  \ \FC_{\sigma_r^j(q_1)} \boxtimes \cdots \boxtimes  \FC_{\sigma_r^j(q_r)}\right)\Big](\sigma_r)\right)}  {ch^*\left(\Big[\Lambda_{-1} N_M M^{r}\Big](\sigma_r) \right)}.$$
Here we identify $X$ and $M$ with the fixed point sets  $(X^r)^{\sigma_r}$ and resp. $(M^r)^{\sigma_r}$ as before, with $N_M M^{r}$ the normal bundle of $M$ in $M^r$, and $ch_X^*$ is a suitable {\it localized} Chern character. Note that by construction the $\langle\sigma_r\rangle$-equivariant vector bundle complex of (\ref{150}) is exact off $X$, i.e., it defines a class $$\Big[\Delta_r^* \Big(  \bigoplus_{j=1}^r  \ \FC_{\sigma_r^j(q_1)} \boxtimes \cdots \boxtimes  \FC_{\sigma_r^j(q_r)} \Big)\Big] \in K_{
\langle\sigma_r\rangle}^0(M,M \setminus X) \simeq K^0(M,M \setminus X) \otimes R({
\langle\sigma_r\rangle}),$$ where $R({
\langle\sigma_r\rangle})$ is the complex representation ring of the cyclic group generated by $\sigma_r$. Then $$(\sigma_r):K_{
\langle\sigma_r\rangle}^0(M,M \setminus X) \simeq K^0(M,M \setminus X) \otimes R({
\langle\sigma_r\rangle}) \to K^0(M,M \setminus X) \otimes \C$$
is induced by taking the trace homomorphism $tr(-;\sigma_r):R({
\langle\sigma_r\rangle})  \to \C$ (see \cite{M}[p.67]).
Finally, we get $$\Big[\Delta_r^* \Big(  \bigoplus_{j=1}^r  \ \FC_{\sigma_r^j(q_1)} \boxtimes \cdots \boxtimes  \FC_{\sigma_r^j(q_r)}\Big)\Big](\sigma_r) =\Big[  \oplus  \FC^{k_1}_{q_1} \otimes \cdots \otimes  \FC^{k_r}_{q_r}  \Big] \otimes tr(\oplus_{j=1}^r  \C; \sigma_r)=0$$
since the corresponding trace of the (signed) cyclic permutation of order $r$  is zero.
 
Together with the additivity of $td_*(-;\sigma_r)$, this yields (\ref{imp}).

\end{proof}

We now have all the ingredients for proving Lemma \ref{l3}.
\begin{proof} 
{\allowdisplaybreaks
\begin{eqnarray*} {T_{(-y)}}_*(\MC^{\boxtimes{r}};\sigma_r) 
&:=&  td_*(-;\sigma_r) \circ {\rm MHC}_{-y}^{\Sigma_r} (\MC^{\boxtimes r}) \\
&:=&   td_*(-;\sigma_r) \left( \sum_p [gr^F_{-p}DR^{\Sigma_r}(\MC^{\boxtimes r})] \cdot y^p \right)\\
&=& \sum_p td_*([gr^F_{-p}DR^{\Sigma_r}(\MC^{\boxtimes r})] ;\sigma_r) \cdot y^p \\
&\overset{(\ref{imp})}{=}& \sum_q \Psi_r td_*([gr^F_{-q}DR(\MC)]) \cdot (y^r)^q\\
&=&\Psi_r td_*\left( \sum_q [gr^F_{-q}DR(\MC)] \cdot y^q  \right) \\
&=& \Psi_r \left( td_*\circ {\rm MHC}_{-y}(\MC) \right)\\
&=& \Psi_r  {T_{(-y)}}_*(\MC).
\end{eqnarray*} 
}
\end{proof}


\section{Comparison with Chern class formulae}\label{r1} In this section we first show how 
Ohmoto's generating series formula (\ref{Oh}) for the MacPherson Chern classes of symmetric products can be derived as a special case of a suitable re-normalization of our motivic Hirzebruch class formula (\ref{te100}). Secondly, we explain how formula (\ref{sh2}) for Chern classes of symmetric powers of a constructible sheaf complex underlying a complex of mixed Hodge modules can similarly be deduced from our Hirzebruch class formula of the main Theorem \ref{t1}.


\subsection{Ohmoto's Chern class formula}

We begin with a general discussion on normalized motivic Hirzebruch classes. 

The power series $Q_y(\alpha)=\frac{\alpha (1+ye^{-\alpha})}{1-e^{-\alpha}} \in \Q[y][[\alpha]]$ mentioned in the introduction is not normalized, as its zero-degree part is $1+y$, instead of $1$. So one can consider the normalized power series \begin{equation}\widehat{Q}_y(\alpha):=\frac{Q_y \left( \alpha (1+y) \right)}{1+y}=\frac{\alpha (1+y)}{1-e^{-\alpha (1+y)}}-\alpha y\end{equation} which defines the  normalized cohomology Hirzebruch class ${\widehat{T}_y}^*(-)$. 
For $y=-1$, this class specializes to the total cohomology Chern class. 

In the singular context, the corresponding normalized motivic Hirzebruch class transformation  $\widehat{T}_{y*}$ is obtained from the transformation ${T_y}_*$  of Section \ref{motHir} by a simple re-normalization procedure (e.g., see \cite{BSY}). More precisely, for a complex algebraic variety $X$ we let $\widehat{T}_{y*}$ be defined by the composition 
\begin{equation}\label{norm}\widehat{T}_{y*} : K_0(var/X)  \overset{ {T_y}_*}{\to} H_{ev}^{BM}(X) \otimes \Q[y] \overset{{\Psi_{(1+y)}} }{\to} H_{ev}^{BM}(X) \otimes \Q[y,(1+y)^{-1}],\end{equation}
with the normalization functor ${\Psi_{(1+y)}}$ given in degree $2k$ by multiplication by $(1+y)^{-k}$. The corresponding normalized motivic Hirzebruch class of the variety $X$ is defined by:
$$\widehat{T}_{y*}(X):=\widehat{T}_{y*}([id_X]).$$

It follows from \cite{BSY} that $\widehat{T}_{y*}(X) \in H_{ev}^{BM}(X) \otimes \Q[y]$. Moreover, by loc. cit., if $y=-1$ one gets that \begin{equation}\label{chern}\widehat{T}_{-1*}(X)=c_*(X) \otimes \Q \end{equation} is the rationalized homology Chern class of MacPherson \cite{MP}. 
\bigskip

For simplicity, the main results of this paper are formulated only in terms of the un-normalized Hirzebruch class transformation. However, for the purpose of comparing our formula (\ref{te100}) with Ohmoto's generating series formula (\ref{Oh}) for the MacPherson-Chern classes of symmetric products of a quasi-projective variety \cite{Oh}, we need to say a few words about the normalized version of our formula  (\ref{te100}).

By applying the normalization functor ${\Psi_{(1-y)}}$ (note that due to our indexing conventions, $y$ is replaced here by $-y$) to the left-hand side of  (\ref{te100}),  we get the generating series $\sum_{n\geq 0} \widehat{T}_{(-y)_*}(X^{(n)})$. Applying the same procedure to the right-hand side of 
 (\ref{te100}), we first note that the normalization functor ${\Psi_{(1-y)}}$ commutes with push-forward for proper maps, as well as with external products, therefore ${\Psi_{(1-y)}}$ also commutes with the Pontrjagin product and the exponential. But ${\Psi_{(1-y)}}\Psi_r{{T_{(-y)}}_*}(X)$ is not in general equal to $\widehat{T}_{-t*}(X)\vert_{t=y^r}$. Only in the case $y=1$ we get the following:
\begin{lemma}\label{l10} With the above notations, the following identification holds:
\begin{equation}\label{chern}
\lim_{y \to 1}  {\Psi_{(1-y)}}\Psi_r{{T_{(-y)}}_*}(X)=\widehat{T}_{-1*}(X)=c_*(X) \otimes \Q . 
\end{equation}
\end{lemma}

The first equality of Lemma \ref{l10} follows by applying the following identity of transformations to the distinguished element $[id_X] \in K_0(var/X)$.

\begin{lemma}\label{l11} With the above notations, the following identification of transformations holds:
\begin{equation}\label{fin}
\lim_{y \to 1}  {\Psi_{(1-y)}}\Psi_r{{T_{(-y)}}_*}(-)=\widehat{T}_{-1*}(-): K_0(var/X) \to  H^{BM}_{ev}(X;\Q) \ .
\end{equation}
\end{lemma}

\begin{proof} 
Since both sides of (\ref{fin}) are defined by functorial group homomorphisms, this identity   can be checked on generators. So, by functoriality for proper push-downs, it suffices to check it in the case when $X$ is smooth. In this case, we have
\begin{equation}\label{234}{{T_{(-y)}}_*}(X)=T^*_{(-y)}(TX) \cap [X]= \Big( \prod_{i=1}^{\dim(X)} Q_{(-y)}(\alpha_i) \Big) \cap [X],\end{equation} with $\alpha_i \in H^2(X;\Q)$ the Chern roots of the tangent bundle $TX$ and $Q_y(\alpha)=\frac{\alpha (1+ye^{-\alpha})}{1-e^{-\alpha}}$ the power series defining the un-normalized cohomological Hirzebruch class $T_y^*(-)$. By applying the $r$-th homological Adams operation to (\ref{234}), we get
\begin{equation}\label{345}\Psi_r {{T_{(-y)}}_*}(X)=\Big( \prod_{i=1}^{\dim(X)} Q_{(-y^r)}(r\alpha_i) \Big) \cap (r^{-\dim(X)}[X])=\Big( \prod_{i=1}^{\dim(X)} \frac{Q_{(-y^r)}(r\alpha_i)}{r} \Big) \cap [X].\end{equation}
Similarly, 
\begin{equation}\label{456}
\begin{split} {\Psi_{(1-y)}}\Psi_r {{T_{(-y)}}_*}(X)&=\Big( \prod_{i=1}^{\dim(X)} \frac{Q_{(-y^r)}(r\alpha_i(1-y))}{r} \Big) \cap \Big((1-y)^{-\dim(X)}[X]\Big)\\
&=\Big( \prod_{i=1}^{\dim(X)} \frac{Q_{(-y^r)}(r\alpha_i(1-y))}{r(1-y)} \Big) \cap [X]\\
&=\Big( \prod_{i=1}^{\dim(X)} \frac{ \alpha_i \left( 1-y^r e^{-r\alpha_i (1-y)} \right) } {1-e^{-r\alpha_i (1-y)}} \Big) \cap [X].
\end{split}
\end{equation}
By l'H\^opital's rule, we have:
\begin{align*}
\lim_{y \to 1} \frac{ \alpha \left( 1-y^r e^{-r\alpha (1-y)} \right) } {1-e^{-r\alpha (1-y)}} =
\lim_{y \to 1} \frac{ \alpha \left( -ry^{r-1}-r\alpha y^r \right)  e^{-r\alpha (1-y)} } { -r \alpha e^{-r\alpha (1-y)}} = \frac{-r-r\alpha}{-r}=1+\alpha.
\end{align*}
Thus, we get:
\begin{equation}\label{fin}
\lim_{y \to 1}  {\Psi_{(1-y)}}\Psi_r{{T_{(-y)}}_*}(X)=\Big( \prod_{i=1}^{\dim(X)} (1+\alpha_i) \Big) \cap [X]=c^*(TX) \cap [X]=c_*(X).\end{equation}

\end{proof}
Therefore, by specializing to $y=1$ in our formula (\ref{te100}), we recover as a corollary Ohmoto's Chern class formula \cite{Oh}:
\begin{cor}
For any quasi-projective complex algebraic variety $X$, the following formula holds in the homology  Pontrjagin ring $PH_*(X)$ with $\Q$-coefficients:
\begin{equation}
\sum_{n \geq 0} {c_*} (X^{(n)} ) \cdot t^n= \exp \left( \sum_{r \geq 1}  d^r_* {c_*}(X)  \cdot \frac{t^r}{r} \right) ,
\end{equation}
with $c_*(-)$ denoting the rationalized homology Chern class of MacPherson.
\end{cor}


\subsection{Chern classes of constructible sheaf complexes}

The above arguments can be extended to obtain as a corollary of our main Theorem \ref{t1}  the generating series formula (\ref{sh2}) for the rationalized MacPherson--Chern classes of symmetric products of a constructible sheaf complex $\FC$ underlying a complex of mixed Hodge modules $\MC$. For mixed Hodge module complexes, the normalized Hirzebruch class transformation $\widehat{T}_{y*}(-)$ is defined as before by the composition 
\begin{equation}\label{norm2}\widehat{T}_{y*} : K_0(\mh(X))  \overset{ {T_y}_*}{\to} H_{ev}^{BM}(X) \otimes \Q[y^{\pm 1}] \overset{{\Psi_{(1+y)}} }{\to} H_{ev}^{BM}(X) \otimes \Q[y^{\pm 1},(1+y)^{-1}].\end{equation} 
As shown in \cite[Proposition 5.21]{Sch-MSRI},  the transformation $\widehat{T}_{y*}$ of (\ref{norm2}) takes in fact values in $H_{ev}^{BM}(Z) \otimes \Q[y^{\pm 1}]$, so in particular one is allowed to specialize the parameter $y$ of the transformation to the value $y=-1$. 

In order to extend the above arguments to the setting of constructible sheaves considered here, we use the commutativity of the following diagram (see \cite{Sch-MSRI}[Proposition 5.21]):
\begin{equation}
\begin{CD} 
K_0(\mh(X)) @>rat>> K_0(D^b_c(X)) \\
@V\widehat{T}_{-1*}VV  @VV\chi_{\rm stalk}V \\
H^{BM}_{ev}(X;\Q) @<<c_* \otimes \Q< F(X)
\end{CD}
\end{equation}
Here, $rat: D^b\mh(X) \to D^b_c(X)$ is the forgetful functor associating to a complex of mixed Hodge modules the underlying constructible sheaf complex, and $\chi_{\rm stalk}$ is defined by taking the Euler characteristics of the stalk complexes. 
Then we have:

\begin{cor} For any quasi-projective complex algebraic variety $X$ and $\FC=rat (\MC)$ the underlying constructible sheaf complex of a complex of mixed Hodge modules $\MC \in D^b\mh(X)$, the following formula holds in the Pontrjagin ring $PH_*(X)$:
\begin{equation}\label{shp}
\sum_{n \geq 0} {c_*} (\FC^{(n)} ) \cdot t^n= \exp \left( \sum_{r \geq 1}  d^r_* {c_*}(\FC)  \cdot \frac{t^r}{r} \right) \ ,
\end{equation}
with $c_*(\FC):=c_*(\chi_{\rm stalk}(\FC))$.
\end{cor}

The proof is exactly the same as above, and is based on the fact that the functor $rat$ commutes with symmetric products (see \cite{MSS}), together with the identification of transformations:
\begin{equation}\label{fin2}
\lim_{y \to 1}  {\Psi_{(1-y)}}\Psi_r{{T_{(-y)}}_*}(-)=\widehat{T}_{-1*}(-): K_0(\mh(X)) \to  H^{BM}_{ev}(X;\Q) 
\end{equation}
which follows from combining (\ref{fin}) with the proof of \cite{Sch-MSRI}[Proposition 5.21].

\bigskip


\bibliographystyle{amsalpha}

\end{document}